\documentclass[reqno]{amsart}

\newtheorem{theorem}{Theorem}[section]
\newtheorem{lemma}[theorem]{Lemma}
\newtheorem{proposition}[theorem]{Proposition}
\newtheorem{corollary}[theorem]{Corollary}

\theoremstyle{definition}

\newtheorem{example}[theorem]{Example}

\newtheorem{remark}[theorem]{Remark}

\usepackage{amscd,amssymb}

\begin{document}

\title[Invariant theory of relatively free algebras]
{Invariant theory of relatively free\\
right-symmetric and Novikov algebras}

\author[Vesselin Drensky]
{Vesselin Drensky}
\address{Institute of Mathematics and Informatics,
Bulgarian Academy of Sciences,
1113 Sofia, Bulgaria}
\email{drensky@math.bas.bg}
\thanks
{Partially supported by Grant I02/18
``Computational and Combinatorial Methods
in Algebra and Applications''
of the Bulgarian National Science Fund.}


\maketitle

\centerline{\it Dedicated to Askar Dzhumadil'daev on the occasion of his 60th birthday}

\begin{abstract}
Algebras with the polynomial identity
$(x_1,x_2,x_3)=(x_1,x_3,x_2)$, where $(x_1,x_2,x_3)=x_1(x_2x_3)-(x_1x_2)x_3$ is the associator, are called right-symmetric.
Novikov algebras are right-symmetric algebras satisfying additionally the polynomial identity $x_1(x_2x_3)=x_2(x_1x_3)$.
We consider the free right-symmetric algebra $F_d({\mathfrak R})$ and the free Novikov algebra $F_d({\mathfrak N})$ freely generated by
$X_d=\{x_1,\ldots,x_d\}$ over a filed $K$ of characteristic 0.
The general linear group $GL_d(K)$ with its canonical action on the $d$-dimensional vector space $KX_d$
acts on $F_d({\mathfrak R})$ and $F_d({\mathfrak N})$ as a group of linear automorphisms.
For a subgroup $G$ of $GL_d(K)$ we study the algebras of $G$-invariants $F_d({\mathfrak R})^G$ and $F_d({\mathfrak N})^G$.
For a large class of groups $G$ we show that the algebras $F_d({\mathfrak R})^G$ and $F_d({\mathfrak N})^G$ are never finitely generated.
The same result holds for any subvariety of the variety $\mathfrak R$ of right-symmetric algebras which contains the subvariety $\mathfrak L$ of
left-nilpotent of class 3 algebras in $\mathfrak R$.
\end{abstract}


\section*{Introduction}
In this paper we fix a field $K$ of characteristic 0 and consider nonassociative $K$-algebras.
An algebra $A$ is called {\it right-symmetric}
if it satisfies the polynomial identity
\begin{equation}\label{right-symmetric identity}
(x_1,x_2,x_3)=(x_1,x_3,x_2),
\end{equation}
where $(x_1,x_2,x_3)=x_1(x_2x_3)-(x_1x_2)x_3$ is the associator, i.e.,
\[
(a_1,a_2,a_3)=(a_1,a_3,a_2)\text{ for all }a_1,a_2,a_3\in A.
\]
A right-symmetric algebra is {\it Novikov} if it satisfies additionally the polynomial identity
of left-commutativity
\begin{equation}\label{left-commutativity}
x_1(x_2x_3)=x_2(x_1x_3).
\end{equation}
We denote by $\mathfrak R$ and $\mathfrak N$ the varieties of all right-symmetric algebras and all Novikov algebras, respectively.
For details on the history of right-symmetric and Novikov algebras we refer to the introductions of the paper
by Dzhumadil'daev and L\"ofwall \cite{DL} and the recent preprint by Bokut, Chen, and Zhang \cite{BCZ}.
The origins of the right-symmetric algebras can be traced back till the paper by Cayley \cite{C} in 1857.
Translated in modern language, Cayley mentioned an identity which implies the right-symmetric identity
for the associators and holds for the right-symmetric Witt algebra in $d$ variables
\[
W_d^{\text{rsym}}=\left\{\sum_{i=1}^df_i\frac{\partial}{\partial x_i}\mid f_i\in K[X_d]\right\}
\]
equipped with the multiplication
\[
\left(f_i\frac{\partial}{\partial x_i}\right)\ast\left(f_j\frac{\partial}{\partial x_j}\right)
=\left(f_j\frac{\partial f_i}{\partial x_j}\right)\frac{\partial}{\partial x_i}.
\]
Cayley also considered the realization of $W_d^{\text{rsym}}$ in terms of rooted trees.
Later right-symmetric algebras were studied under different names: Vinberg, Koszul, Gerstenhaber, and pre-Lie algebras,
see the references in \cite{DL}.
The opposite algebras of Novikov algebras (satisfying the left-symmetric identity for the associators
and right commutativity) appeared in the paper by Gel'fand and Dorfman \cite{GD}. There the authors gave
an algebraic approach to the notion of Hamiltonian
operator in finite-dimensional mechanics and the formal calculus of
variations. Independently, later Novikov algebras were rediscovered by Balinskii and Novikov
in the study of equations of hydrodynamics \cite{BN}, see also the survey article by Novikov \cite{N}.
(Due to the contributions in \cite{GD, BN} and \cite{N} Bokut, Chen, and Zhang \cite{BCZ} suggest to call these algebras
Gel'fand-Dorfman-Novikov algebras. We shall continue to keep the name Novikov algebras.)
An example of a Novikov algebra is the right-symmetric Witt algebra $W_1^{\text{rsym}}$ in one variable.
In a series of papers, see, e.g., \cite{DL, D, DI} Dzhumadil'daev,
with coauthors or alone, has studied free right-symmetric and free Novikov algebras,
with applications to nonassociative algebras with polynomial identities.

In commutative invariant theory one usually considers the general linear group $GL_d(K)$ with its canonical action on the $d$-dimensional vector space
$V_d$ with basis $\{e_1,\ldots,e_d\}$. This induces an action on the polynomial algebra $K[X_d]=K[x_1,\ldots,x_d]$ in $d$ variables
\[
g(f(v))=f(g^{-1}(v)),\quad g\in GL_d(K), v\in V_d,
\]
where the linear functions $x_i:V_d\to K$ are defined by
\[
x_i(e_j)=\delta_{ij},\quad i,j=1,\ldots,d,
\]
and $\delta_{ij}$ is the Kronecker symbol. For our noncommutative considerations it is more convenient to suppress one step and,
replacing $V$ with its dual space $V^{\ast}$, to assume that $GL_d(K)$ acts canonically on the vector space
$KX_d$ with basis $X_d=\{x_1,\ldots,x_d\}$. Then, identifying the polynomial algebra $K[X_d]$ with the symmetric algebra of $KX_d$,
we extend diagonally this action of $GL_d(K)$ on $K[X_d]$:
\begin{equation}\label{diagonal action}
g(f(X_d))=g(f(x_1,\ldots,x_d))=f(g(x_1),\ldots,g(x_d)),
\end{equation}
$g\in GL_d(K)$, $f(X_d)\in K[X_d]$.
In this way $GL_d(K)$ acts as the group of linear automorphisms of $K[X_d]$.
For a subgroup $G$ of $GL_d(K)$ {\it the algebra of $G$-invariants} is
\[
K[X_d]^G=\{f\in K[X_d]\mid g(f)=f\text{ for all } g\in G\}.
\]
This is a $\mathbb Z$-graded vector space and its {\it Hilbert} (or {\it Poincar\'e}) {\it series} is the formal power series
\[
H(K[X_d]^G,z)=\sum_{n\geq 0}\dim(K[X_d]^G)_nz^n,
\]
where $(K[X_d]^G)_n$ is the homogeneous component of degree $n$ in $K[X_d]^G$.
The following are among the main problems related with the description of the algebra $K[X_d]^G$ for different groups or classes of groups $G$.
For details concerning also computational and algorithmic problems see \cite{DeK} or \cite{St}.

\begin{itemize}
\item {\it Is the algebra $K[X_d]^G$ finitely generated?} This problem was the main motivation for the Hilbert
14th problem in his famous lecture {\it ``Mathematische Probleme''} given at the
International Congress of Mathematicians held in 1900 in Paris \cite{H}.
It is known that $K[X_d]^G$ is finitely generated for finite groups (the theorem of Emmy Noether \cite{No}),
for reductive groups (the Hilbert-Nagata theorem, see e.g., \cite{DiC}), and for groups close to reductive (see e.g., Grosshans \cite{Gr}
and Had\v{z}iev \cite{Ha}).
The first example of an algebra of invariants $K[X_d]^G$ which is not finitely generated is due to Nagata \cite{Na}.

\item {\it If $K[X_d]^G$ is finitely generated, describe it in terms of generators and defining relations.} In different degree of generality
this problem is solved for classes of groups. For example, the theorem of Emmy Noether \cite{No} gives that for finite groups the algebra $K[X_d]^G$
is generated by invariants of degree $\leq \vert G\vert$. Also for finite groups, the Chevalley-Shephard-Todd theorem
\cite{Ch, ST} states that
the algebra $K[X_d]^G$ is isomorphic to the polynomial algebra in $d$ variables (i.e., it is generated by a set of $d$ algebraically independent
invariants) if and only if $G$ is generated by pseudo-reflections.

\item{\it Calculate the Hilbert series $H(K[X_d]^G,z)$.} For finite groups the answer is given by the Molien formula \cite{M}
\[
H(K[X_d]^G,z)=\frac{1}{\vert G\vert}\sum_{g\in G}\frac{1}{\det(1-gz)}.
\]
The analogue for reductive and close to them groups is the Molien-Weyl integral formula \cite{W1}, see also \cite{W2}.
\end{itemize}

In noncommutative invariant theory one replaces the polynomial algebras $K[X_d]$ with other noncommutative or nonassociative algebras
still keeping some of the typical features of polynomial algebras. One such feature is the universal property that for an arbitrary
commutative algebra $A$ every mapping $X_d\to A$ is extended to a homomorphism $K[X_d]\to A$. In the noncommutative set-up the class of
commutative algebras is replaced by an arbitrary variety of algebras $\mathfrak V$ and instead on $K[X_d]^G$ one studies the
algebra of $G$-invariants $F_d({\mathfrak V})^G$ of the $d$-generated relatively free algebra $F_d({\mathfrak V})$ in $\mathfrak V$, $d\geq 2$.
For a background see the surveys \cite{F, Dr2}.
Comparing with commutative invariant theory, when $K[X_d]^G$ is finitely generated for all ``nice'' groups,
the main difference in the noncommutative case is that $F_d({\mathfrak V})^G$ is finitely generated quite rarely.
For a survey on invariants of finite groups $G$ acting on relatively free associative algebras
see \cite{F, Dr2} and \cite{KS}. For finite groups $G\not=\langle 1\rangle$
and varieties of Lie algebras $F_d({\mathfrak V})^G$
is finitely generated if and only if $\mathfrak V$ is nilpotent, see \cite{Br, Dr1}.

Concerning the Hilbert series of $F_d({\mathfrak V})^G$, for $G$ finite there is an analogue of the Molien formula, see Formanek \cite{F}.
Let
\[
H(F_d({\mathfrak V}),z_1,\ldots,z_d)=\sum_{n_i\geq 0}\dim F_d({\mathfrak V})_{(n_1,\ldots,n_d)}z_1^{n_1}\cdots z_d^{n_d}
\]
be the Hilbert series of $F_d({\mathfrak V})$ as a multigraded vector space. It is equal to the generating function of the
dimensions of the vector spaces $F_d({\mathfrak V})_{(n_1,\ldots,n_d)}$ of the elements in $F_d({\mathfrak V})$ which are homogeneous of degree $n_i$
in $x_i$.
If $\xi_1(g),\ldots,\xi_d(g)$ are the eigenvalues of $g \in G$, then
the Hilbert series of the algebra of invariants $F_d({\mathfrak V})^G$ is
\[
H(F_d({\mathfrak V})^G;z) = \frac{1}{\vert G\vert}\sum_{g\in G}H(F_d({\mathfrak V});\xi_1(g)z,\ldots,\xi_d(g)z).
\]
There is also an analogue of the Molien-Weyl formula for the Hilbert series of $F_d({\mathfrak V})^G$ which combines ideas of
De Concini, Eisenbud, and Procesi \cite{DEP} and
Almkvist, Dicks, and Formanek \cite{ADF}. Evaluating the corresponding multiple integral one uses the Hilbert series of $F_d({\mathfrak V})$
instead of the Hilbert series of $K[X_d]$
\[
H(K[X_d],z_1,\ldots,z_d)=\prod_{i=1}^d\frac{1}{1-z_i}.
\]
We refer to \cite{BBDrGK} for other methods for computing the Hilbert series
of $F_d({\mathfrak V})^G$ when $G$ is isomorphic to the special linear group $SL_m(K)$ or to the group $UT_m(K)$ of the $m\times m$ unitriangular matrices.

In this paper we study invariant theory of relatively free right-symmetric and Novikov algebras. Let $\mathfrak L$ be the variety
of right-symmetric algebras which are left-nilpotent of class 3, i.e., $\mathfrak L$ is the subvariety of $\mathfrak R$ satisfying
the polynomial identity
\begin{equation}\label{left nilpotency}
x_1(x_2x_3)=0.
\end{equation}
For a large class of subgroups $G$ of $GL_d(K)$, $G\not=\langle 1\rangle$, $d>1$, we show that
$F_d({\mathfrak V})^G$ is not finitely generated for any variety $\mathfrak V$ containing $\mathfrak L$. More precisely, let
$A_d=K[X_d]_+$ be the algebra of polynomials without constant term and let $(A_d)_1^G=(KX_d)^G$ be the vector space of linear polynomials fixed by $G$.
Clearly, $(A_d^2)^G$ is a $K[(A_d)_1^G]$-module. If $(A_d^2)^G$ is not finitely generated as a $K[(A_d)_1^G]$-module, then
$F_d({\mathfrak V})^G$ is not finitely generated for any $\mathfrak V$ containing $\mathfrak L$. The class of such groups $G$ contains all finite groups.
It contains also the classical and close to them groups under some natural restrictions on the embedding into $GL_d(K)$.
In particular, if $(A_d)_1^G=0$ and $(A_d^2)^G\not=0$, then $F_d({\mathfrak V})^G$ is not finitely generated.
Results in the same spirit hold if we replace the polynomial algebra $K[X_d]$ with the free metabelian Lie algebra
$F_d({\mathfrak A}^2)=L_d/L_d''$, where $L_d$ is the free Lie algebra freely generated by $X_d$ and ${\mathfrak A}^2$ is the variety of all metabelian
(solvable of class 2) Lie algebras. If $(KX_d)^G=(A_d)_1^G=0$ and $\dim F_d({\mathfrak A}^2)^G=\infty$,
then again $F_d({\mathfrak V})^G$ is not finitely generated.

\section{Preliminaries}
We fix a field $K$ of characteristic 0. All vector spaces and algebras will be over $K$.
Let
\[
F(X)=K\{X\}=K\{x_1,x_2,\ldots\}
\]
be the (absolutely) free nonassociative algebra freely generated by the countable set $X=\{x_1,x_2,\ldots\}$.
Recall that the polynomial $f(x_1,\ldots,x_m)\in K\{X\}$
is a polynomial identity for the algebra $A$ if
$f(a_1,\ldots,a_m)=0$ for all $a_1,\ldots,a_m\in A$.
The class of all algebras satisfying a given set
$U\subset K\{X\}$
of polynomial identities is called the variety of associative algebras
defined by the system $U$.
If $\mathfrak V$ is a variety, then $T({\mathfrak V})$ is the ideal
of $K\{X\}$ consisting of all polynomial identities of $\mathfrak V$. Let $X_d=\{x_1,\ldots,x_d\}\subset X$. Then
the algebra
\[
F_d({\mathfrak V})=K\{x_1,\ldots,x_d\}/
(K\{x_1,\ldots,x_d\}\cap T({\mathfrak V}))=K\{X_d\}/
(K\{X_d\}\cap T({\mathfrak V}))
\]
is the relatively free algebra of rank $d$ in $\mathfrak V$.
We shall denote the generators of $F_d({\mathfrak V})$ with the same symbols $X_d$.
The ideals $K\{X_d\}\cap T({\mathfrak V})$ of
$K\{X_d\}$ are preserved by all endomorphisms $\varphi$
of $K\{X_d\}$, i.e.,
$\varphi(K\{X_d\}\cap T({\mathfrak V}))\subseteq K\{X_d\}\cap T({\mathfrak V})$.
In particular, $GL_d(K)(K\{X_d\}\cap T({\mathfrak V}))=K\{X_d\}\cap T({\mathfrak V})$.
Here the general linear group $GL_d(K)$ acts canonically
on the vector space $KX_d$ with basis $X_d$ and this action is extended diagonally on the whole $F_d({\mathfrak V})$ as in (\ref{diagonal action}).
Hence $F_d({\mathfrak V})$ has a natural structure of a $GL_d(K)$-module.
For a background on representation theory of $GL_d(K)$ see, e.g., \cite{Mc, W2}.
Since $\text{char}(K)=0$, the algebra $F_d({\mathfrak V})$ is a direct sum
of irreducible $GL_d(K)$-modules and
\[
F_d({\mathfrak V})=\sum m_{\lambda}({\mathfrak V})W_d(\lambda),
\]
where $W_d(\lambda)$ is the irreducible polynomial $GL_d(K)$-module corresponding to the partition
$\lambda=(\lambda_1,\ldots,\lambda_d)$, $\lambda_1\geq\cdots\geq\lambda_d\geq 0$, and $m_{\lambda}({\mathfrak V})$ is the multiplicity of
$W_d(\lambda)$ in the decomposition of $F_d({\mathfrak V})$. Then the Hilbert series of $F_d({\mathfrak V})$ is
\[
H(F_d({\mathfrak V}),z_1,\ldots,z_d)=\sum m_{\lambda}({\mathfrak V})S_{\lambda}(z_1,\ldots,z_d),
\]
where $S_{\lambda}(z_1,\ldots,z_d)$ is the Schur function corresponding to $\lambda$.
Since the Schur functions form a basis of the vector space $K[X_d]^{S_n}$ of symmetric polynomials in $d$ variables,
the Hilbert series $H(F_d({\mathfrak V}),z_1,\ldots,z_d)$ determines the $GL_d(K)$-module structure of $F_d({\mathfrak V})$.

In the sequel we shall need some well known information for two relatively free algebras: the polynomial algebra
$K[X_d]$ and the free metabelian Lie algebra $F_d({\mathfrak A}^2)=L_d/L_d''$.

\begin{lemma}\label{polynomial and metabelian Lie algebra}
{\rm (i)} The $GL_d(K)$-module structure of the polynomial algebra $K[X_d]$ is
\[
K[X_d]=\sum_{n\geq 0}W_d(n).
\]

{\rm (ii)} The free metabelian Lie algebra $F_d({\mathfrak A}^2)$ has a basis
\[
\{x_i,[[\cdots[x_{i_1},x_{i_2}],\ldots],x_{i_n}]\mid i,i_j=1,\ldots,d,i_1>i_2\leq \cdots\leq i_n\}.
\]
The $GL_d(K)$-module structure of $F_d({\mathfrak A}^2)$ is
\[
F_d({\mathfrak A}^2)=W_d(1)+\sum_{n\geq 2}W_d(n-1,1).
\]
\end{lemma}

Part (i) of the lemma is well known. Part (ii) is also well known, see e.g.,
\cite[\S 52, pp. 274-276 of the English translation]{Sh} for the basis of $F_d({\mathfrak A}^2)$
and \cite[the proof of Lemma 2.5]{DrK} for its $GL_d(K)$-module structure.

The product of two Schur functions $S_{\lambda}(z_1,\ldots,z_d)S_{\mu}(z_1,\ldots,z_d)$
can be expressed as a sum of Schur functions using the Littlewood-Richardson rule. A very special case of this rule
is the Branching theorem, when $\mu=(1)$. It states that
\begin{equation}\label{Branching theorem}
S_{\lambda}(z_1,\ldots,z_d)S_{(1)}(z_1,\ldots,z_d)=\sum S_{\nu}(z_1,\ldots,z_d),
\end{equation}
where the sum runs on all partitions $\nu=(\nu_1,\ldots,\nu_d)$ obtained by adding 1
to one of the components $\lambda_i$ of $\lambda=(\lambda_1,\ldots,\lambda_d)$. In other words, the Young diagram of $\nu$ is obtained
by adding a box to the diagram of $\lambda$. Since the product of two Schur functions corresponds to the tensor product of the corresponding
irreducible $GL_d(K)$-modules, we obtain equivalently
\begin{equation}\label{Branching theorem for modules}
W_d(\lambda)\otimes_K W_d(1)=\sum W_d(\nu),
\end{equation}
with the same summation on $\nu$ as in (\ref{Branching theorem}).

If $G$ is a subgroup of $GL_d(K)$, then the $GL_d(K)$-action on the irreducible $GL_d(K)$-module $W_d(\lambda)$ induces a  $G$-action on $W_d(\lambda)$.
Let $W_d(\lambda)^G$ be the vector space of the elements of $W_d(\lambda)$ fixed by $G$, i.e., of the $G$-invariants of $W_d(\lambda)$.
If $W$ is a graded $GL_d(K)$-module with polynomial homogeneous components,
\begin{equation}\label{graded GL-module}
W=\bigoplus_{k\geq 0}W_k,\quad W_k=\sum_{\lambda} m_{\lambda}(k)W_d(\lambda),
\end{equation}
then its Hilbert series is
\begin{equation}\label{Hilbert series of graded GL-module}
\begin{split}
H(W,z_1,\ldots,z_d,z)&=\sum_{k\geq 0}\left(\sum_{n_i\geq 0}\dim(W_k)_{(n_1,\ldots,n_d)}z_1^{n_1}\cdots z_d^{n_d}\right)z^k\\
&=\sum_{k\geq 0}\sum_{\lambda} m_{\lambda}(k)S_{\lambda}(z_1,\ldots,z_d)z^k.\\
\end{split}
\end{equation}

\begin{lemma}\label{Hilbert series of F determines this of G-invariants}
Let $W$ be a graded $GL_d(K)$-module with polynomial homogeneous components, as in (\ref{graded GL-module}), and let $G$ be a subgroup of
$GL_d(K)$. Then the Hilbert series of the $G$-invariants of $W$
\[
H(W^G,z)=\sum_{k\geq 0}\dim W_k^Gz^k
\]
is determined from the Hilbert series (\ref{Hilbert series of graded GL-module}) of $W$.
\end{lemma}

\begin{proof}
We follow the main ideas of the recent preprint \cite{DoDr} which contains more applications in the spirit of the lemma.
Since the dimension of $W_d(\lambda)^G$ depends on $W_d(\lambda)$ only, and the Schur functions $S_{\lambda}(z_1,\ldots,z_d)$
are in 1-1 correspondence with the modules $W_d(\lambda)$, we conclude that $\dim W_d(\lambda)^G$ is a function of $S_{\lambda}(z_1,\ldots,z_d)$.
This immediately completes the proof because
\[
H(W^G,z)=\sum_{k\geq 0}\left(\sum_{\lambda}m_{\lambda}(k)\dim W_d(\lambda)^G\right)z^k.
\]
\end{proof}

The proof of the next statement can be found in \cite[Proposition 4.2]{DrG} in the case of homomorphic images of the free associative algebra
$K\langle X_d\rangle$. The proof in the case below is exactly the same.

\begin{proposition}\label{lifting invariants}
Let $I$ be an ideal of the relatively free algebra $F_d({\mathfrak W})$ of the variety $\mathfrak W$
and let $I$ be preserved under the $GL_d(K)$-action on $F_d({\mathfrak W})$. If $G$ is a subgroup of $GL_d(K)$, then every $G$-invariant
of the factor algebra $F_d({\mathfrak W})/I$ can be lifted to a $G$-invariant of $F_d({\mathfrak W})$, i.e., under the canonical homomorphism
\[
\pi:F_d({\mathfrak W})\to F_d({\mathfrak W})/I
\]
$F_d({\mathfrak W})^G$ maps onto $(F_d({\mathfrak W})/I)^G$. In particular, if
$\mathfrak V$ is a subvariety of the variety $\mathfrak W$ and $\pi:F_d({\mathfrak W})\to F_d({\mathfrak V})$, then
\[
\pi(F_d({\mathfrak W})^G)=F_d({\mathfrak V})^G.
\]
\end{proposition}

For more details on varieties of algebras (in the associative case) and the applications of representation theory of $GL_d(K)$
to PI-algebras we refer to the book \cite{Dr3}.

\section{The main result}

Let $\mathfrak L$ be the subvariety of the variety of right-symmetric algebras $\mathfrak R$ defined by the identity (\ref{left nilpotency})
of left-nilpotency of class 3. Since the identity (\ref{left-commutativity}) of left-commutativity is a consequence of (\ref{left nilpotency}),
$\mathfrak L$ is also a subvariety of the variety $\mathfrak N$ of Novikov algebras. Working in $\mathfrak L$, the only nonzero products
are left-normed. We shall omit the parentheses and shall write $a_1a_2\cdots a_n$ instead of $(\cdots(a_1a_2)\cdots )a_n$
and $a_1a_2^k$ instead of $a_1\underbrace{a_2\cdots a_2}_{k\text{ times}}$.

\begin{lemma}\label{description of free algebras in L}
{\rm (i)} The relatively free algebra $F_d({\mathfrak L})$ has a basis
\begin{equation}\label{basis of free algebras in L}
\{x_{i_1}x_{i_2}\cdots x_{i_n}\mid i_1=1,\ldots,d,\quad 1\leq i_2\leq\cdots\leq i_n\leq d\}.
\end{equation}
{\rm (ii)} The $GL_d(K)$-module structure of $F_d({\mathfrak L})$ is
\begin{equation}\label{GL-module structure of free algebras in L}
F_d({\mathfrak L})=W_d(1)+\sum_{n\geq 2}(W_d(n)+W_d(n-1,1)).
\end{equation}
\end{lemma}

\begin{proof}
(i) Modulo the identity (\ref{left nilpotency}) the right-symmetric identity (\ref{right-symmetric identity}) reduces to
\begin{equation}\label{RS mod LN}
x_1x_2x_3=x_1x_3x_2.
\end{equation}
Hence $\mathfrak L$ satisfies the identity
\[
x_1x_2\cdots x_n=x_1x_{\sigma(2)}\cdots x_{\sigma(n)},\quad \sigma\in S_n,\sigma(1)=1,
\]
and the algebra $F_d({\mathfrak L})$ is spanned as a vector space on the elements (\ref{basis of free algebras in L}).
In order to show that (\ref{basis of free algebras in L}) is a basis of $F_d({\mathfrak L})$ it is sufficient to construct
an algebra $A$ in $\mathfrak L$ which is generated by $a_1,\ldots,a_d$ and has a basis
\begin{equation}\label{the algebra A in L}
\{a_ia_1^{n_1}\cdots a_d^{n_d}\mid i_1=1,\ldots,d, \quad n_j\geq 0\}.
\end{equation}
Since $A$ is a homomorphic image of $F_d({\mathfrak L})$, this would imply
that (\ref{basis of free algebras in L}) is a basis of $F_d({\mathfrak L})$. Consider the vector space $A$
with basis (\ref{the algebra A in L}) and define a multiplication there by the rule
\[
(a_ia_1^{n_1}\cdots a_d^{n_d})\ast a_j=a_ia_1^{n_1}\cdots a_j^{n_j+1}\cdots a_d^{n_d},
\]
\[
(a_ia_1^{n_1}\cdots a_d^{n_d})*(a_ia_1^{m_1}\cdots a_d^{m_d})=0,\text{ if }m_j>0\text{ for some }j.
\]
Obviously $A$ satisfies the identities (\ref{left nilpotency}) and (\ref{RS mod LN}), and hence belongs to $\mathfrak L$.

(ii) For $n\geq 2$ we divide the basis elements from (\ref{basis of free algebras in L}) in two groups.
The first group contains the monomials $x_{i_1}x_{i_2}\cdots x_{i_n}$ with $i_1\leq i_2$ and the second group the monomials
with $i_1>i_2$. Obviously, the monomials in the first group are in 1-1 correspondence with the monomials of degree $\geq 2$ in $K[X_d]$.
By Lemma \ref{polynomial and metabelian Lie algebra} (ii), the same holds for the monomials from the second group and the elements
of degree $\geq 2$ in $F_d({\mathfrak A}^2)$. Hence the Hilbert series of $F_d({\mathfrak L})$ is a sum of the Hilbert series
of the algebra of polynomials without constant term and the commutator ideal of the Lie algebra $F_d({\mathfrak A}^2)$.
Now the proof follows from Lemma \ref{polynomial and metabelian Lie algebra}.
\end{proof}

The construction in the proof of Lemma \ref{description of free algebras in L} (i) suggests that the algebra $F_d({\mathfrak L})$
has the structure of a right $K[X_d]$-module with action defined by
\[
(x_px_1^{n_1}\cdots x_d{n_d})\circ(x_1^{m_1}\cdots x_d^{m_d})=x_px_1^{n_1+m_1}\cdots x_d^{n_d+m_d},\quad p=1,\ldots,d,n_j,m_j\geq 0.
\]
Clearly, the ideal $F_d^2({\mathfrak L})$ of the elements in $F_d({\mathfrak L})$ without linear term is a $K[X_d]$-submodule.
We shall denote by $(A_d)_1$ the vector space $KX_d$ and shall identify $K[X_d]$ and $K[(A_d)_1]$.

The following theorem and its consequences together with the examples in the next section are the main results of the paper.

\begin{theorem}\label{main theorem}
Let $\mathfrak V$ be a subvariety of the variety $\mathfrak R$ of all right-symmetric algebras and let $\mathfrak V$ contain the variety
$\mathfrak L$ of left-nilpotent of class $3$ algebras in $\mathfrak R$. If $G\not=\langle 1\rangle$ is a subgroup of $GL_d(K)$ such that the ideal
$F_d^2({\mathfrak L})^G$ of the algebra of invariants $F_d({\mathfrak L})^G$ is not finitely generated as a $K[(A_d)_1^G]$-module, then
the algebra of $G$-invariants $F_d({\mathfrak V})^G$ is not finitely generated.
\end{theorem}

\begin{proof}
By Proposition \ref{lifting invariants} the canonical homomorphism $F_d({\mathfrak V})\to F_d({\mathfrak L})$ maps
$F_d({\mathfrak V})^G$ onto $F_d({\mathfrak L})^G$ and if $F_d({\mathfrak V})^G$ is finitely generated, the same is $F_d({\mathfrak L})^G$.
Hence it is sufficient to show that $F_d({\mathfrak L})^G$ is not finitely generated. Therefore we may work in $F_d({\mathfrak L})$ and assume that
$F_d({\mathfrak L})^G$ is finitely generated.
As a vector space $F_d({\mathfrak L})^G$ is a direct sum of the invariants of first degree $(KX_d)^G=(A_d)_1^G$ and the invariants $F_d^2({\mathfrak L})^G$
without linear term. We may assume that $F_d({\mathfrak L})^G$ is generated by $U=\{u_1,\ldots,u_k\}\subset (A_d)_1^G$ and
$W=\{w_1,\ldots,w_l\}\subset F_d^2({\mathfrak L})^G$.
Since $F_d({\mathfrak L})^GF_d^2({\mathfrak L})^G=0$, the only nonzero products of the generators of $F_d({\mathfrak L})^G$ are
$u_pu_{i_1}\cdots u_{u_m}$, and $w_qu_{i_1}\cdots u_{u_m}$, $m\geq 0$. Hence $KU=(A_d)_1^G$,
\[
F_d^2({\mathfrak L})^G=\sum_{i=1}^ku_{p_1}u_{p_2}\circ K[(A_d)_1^G]+\sum_{j=1}^lw_q\circ K[(A_d)_1^G]
\]
and $F_d^2({\mathfrak L})^G$ is a finitely generated $K[(A_d)_1^G]$-module which is a contradiction.
\end{proof}

\begin{corollary}\label{reduction to polynomials}
Let $A_d=K[X_d]_+$ be the algebra of polynomials without constant term
and let $G$ be a subgroup of $GL_d(K)$. If $(A_d^2)^G$ is not finitely generated as a $K[(A_d)_1^G]$-module, then
$F_d({\mathfrak V})^G$ is not finitely generated for any variety $\mathfrak V$ containing $\mathfrak L$.
\end{corollary}

\begin{proof}
By the Branching theorem (\ref{Branching theorem for modules})
\begin{equation}\label{products of GL-modules in F}
W_d(n-1,1)\otimes_K W_d(1)=W_d(n,1)\oplus W_d(n-1,2)\oplus W_d(n-1,1,1).
\end{equation}
Consider the $GL_d(K)$-module decomposition of $F_d({\mathfrak L})$ given in
Lemma \ref{description of free algebras in L} (ii). Since $F_d({\mathfrak L})F_d^2({\mathfrak L})=0$, the only nonzero products
$W_d(\lambda)W_d(\mu)$ with $\lambda$ or $\mu$ equal to $(n-1,1)$, $n\geq 2$, come from
\[
W_d(n-1,1)W_d(1)=W_d(n-1,1)F_d({\mathfrak L})=W_d(n-1,1)(KX_d)=W_d(n-1,1)(A_d)_1.
\]
This is a homomorphic image in $F_d({\mathfrak L})$ of $W_d(n-1,1)\otimes_K W_d(1)$. By (\ref{products of GL-modules in F}) we derive that
$W_d(n-1,1)F_d({\mathfrak L})\subset W_d(n,1)$. This implies that
\[
I=\sum_{n\geq 2}W_d(n-1,1)\subset F_d({\mathfrak L})
\]
is an ideal of $F_d({\mathfrak L})$ and the $GL_d(K)$-module structure of the factor algebra is
\[
F_d({\mathfrak L})/I=\sum_{n\geq 1}W_d(n)\cong A_d.
\]
Hence the algebras $F_d({\mathfrak L})/I$ and $A_d$ have the same Hilbert series and
by Lemma \ref{Hilbert series of F determines this of G-invariants} the same holds for their algebras of invariants.
Since $(A_d^2)^G$ is not finitely generated as a $K[(A_d)_1^G]$-module, the same is true for the $K[(A_d)_1^G]$-module
$F_d^2({\mathfrak L})/I$. By Proposition \ref{lifting invariants}, the $K[(A_d)_1^G]$-module $F_d^2({\mathfrak L})$
is not finitely generated and the application of Theorem \ref{main theorem} completes the proof.
\end{proof}

\begin{corollary}\label{metabelian Lie algebras}
Let $F_d({\mathfrak A}^2)$ be the free metabelian Lie algebra and let $G$ be a subgroup of $GL_d(K)$. If
$(KX_d)^G=(A_d)_1^G=0$ and $\dim F_d({\mathfrak A}^2)^G=\infty$,
then $F_d({\mathfrak V})^G$ is not finitely generated for any variety $\mathfrak V$ containing $\mathfrak L$.
\end{corollary}

\begin{proof}
Since $(KX_d)^G=(A_d)_1^G=0$ we obtain that $F_d({\mathfrak L})^G=F_d^2({\mathfrak L})^G$. Hence the algebra $F_d({\mathfrak L})^G$
is with trivial multiplication and the finite generation is equivalent to the finite dimensionality.
As a $GL_d(K)$-module $F_d({\mathfrak A}^2)$ is a homomorphic image of $F_d({\mathfrak L})$. Hence the vector space $F_d({\mathfrak A}^2)^G$
is a homomorphic image of $F_d({\mathfrak L})^G$. This implies that $\dim F_d^2({\mathfrak L})^G=\infty$, i.e., both the algebras
$F_d({\mathfrak L})^G$ and $F_d({\mathfrak V})^G$ are not finitely generated.
\end{proof}

\begin{remark}
In Corollary \ref{metabelian Lie algebras} we cannot remove directly the restriction $(KX_d)^G=0$, as in Corollary \ref{reduction to polynomials},
because the $GL_d(K)$-submodule $I=\sum_{n\geq 2}W_d(n)$ of $F_d({\mathfrak L})$ is not an ideal. For example, one can show that
$W_d(2)(KX_d)=W_d(3)\oplus W_d(2,1)$. Hence we cannot use the property that the Lie algebra $F_d({\mathfrak A}^2)^G$ is not finitely generated
to show that the algebra $F_d({\mathfrak L})$ is also not finitely generated. On the other hand, we do not know examples
of groups $G$ when $(KX_d)^G=0$, $K[X_d]^G=K$, and $\dim F_d({\mathfrak A}^2)^G=\infty$. Such an example would show that we may apply
Corollary \ref{metabelian Lie algebras} when we cannot apply Corollary \ref{reduction to polynomials}.
\end{remark}

\section{Examples}
All examples in this section use the following statement which is a consequence of Corollary \ref{reduction to polynomials}.

\begin{proposition}\label{base of examples}
If for a subgroup $G$ of $GL_d(K)$
\[
\text{\rm transcend.deg}(K[X_d]^G)>\dim(KX_d)^G,
\]
then the algebra $F_d({\mathfrak V})^G$ is not finitely generated for any variety $\mathfrak V$ containing $\mathfrak L$.
\end{proposition}

\begin{proof}
Let $t=\text{\rm transcend.deg}(K[X_d]^G)$. Since $K[X_d]^G$ is graded, we may choose $t$ algebraically independent homogeneous elements in $A_d^G=(K[X_d]_+)^G$.
If $m=\dim(KX_d)^G$, changing linearly the variables $X_d$ we assume that
$(KX_d)^G$ has a basis $X_m=\{x_1,\ldots,x_m\}$ and $K[(A_d)_1^G]=K[X_m]$. Since $t>m$, we obtain that $((A_d)^2)^G$ contains an element $f(X_d)$, such that
the system $X_m\cup\{f(X_d)\}$ is algebraically independent. Hence the $K[X_m]$-module generated by the powers $f^k(X_d)$, $k=1,2,\ldots$, is not finitely generated.
Now the proof follows from Corollary \ref{reduction to polynomials}.
\end{proof}

\subsection{Finite groups}
\begin{theorem}\label{finite groups}
Let $G$ be a finite subgroup of $GL_d(K)$ and $G\not=\langle 1\rangle$. Then
the algebra $F_d({\mathfrak V})^G$ is not finitely generated for any variety $\mathfrak V$ containing $\mathfrak L$.
\end{theorem}

\begin{proof}
It is well known that for a finite group $G$
\begin{equation}\label{transcendence degree}
\text{\rm transcend.deg}(K[X_d]^G)=\text{\rm transcend.deg}(K[X_d])=d.
\end{equation}
For self-containedness of the exposition,
every element $f(X_d)\in K[X_d]$ satisfies the equation
\[
u_f(z)=\prod_{g\in G}(z-g(f(X_d)))=z^{\vert G\vert}-c_1z^{\vert G\vert-1}+c_2z^{\vert G\vert-2}-\cdots\pm c_{\vert G\vert}
\]
where the coefficients $c_k$ are equal to the elementary symmetric polynomials in $\{g(f(X_d))\mid g\in G\}$.
Hence $c_k\in K[X_d]^G$ and as a $K[X_d]^G$-module $K[X_d]$ is generated by
\[
x_1^{a_1}\cdots x_d^{a_d},\quad 0\leq a_i< \vert G\vert.
\]
The finite generation of the $K[X_d]^G$-module $K[X_d]$ implies (\ref{transcendence degree}) and the theorem follows from
Proposition \ref{base of examples}.
\end{proof}

\subsection{Reductive groups}
If $G\subset GL_d(K)$ is a reductive group then there exists a $G$-submodule $W$ of $KX_d$ such that $KX_d=(KX_d)^G\oplus W$.

\begin{proposition}\label{reductive groups}
In the above notation, if $K[W]^G\not=K$, then $F_d({\mathfrak V})^G$ is not finitely generated for all $\mathfrak V$ containing $\mathfrak L$.
\end{proposition}

\begin{proof}
Since the elements of $K[W]$ cannot be expressed as polynomials in $(KX_d)^G$, the condition $K[W]^G\not=K$ implies that
\[
\text{\rm transcend.deg}(K[X_d]^G)=\dim(KX_d)^G+\text{\rm transcend.deg}(K[W]^G)>\dim(KX_d)^G
\]
and this completes the proof in virtue of Proposition \ref{base of examples}.
\end{proof}

\begin{example} For each $k\geq 1$ there is a unique irreducible rational $k$-dimensional $SL_2(K)$-module $W_k$.
Let the subgroup $G$ of $GL_d(K)$ be isomorphic to $SL_2(K)$ and
\[
KX_d\cong W_{k_1}\oplus\cdots\oplus W_{k_p}
\]
as an $SL_2(K)$-module. It is well known that if $k_1\geq 3$, then $K[W_{k_1}]$ contains nontrivial $SL_2(K)$-invariants.
Similarly, $K[W_2\oplus W_2]^{SL_2(K)}\not=K$. Hence the only cases when $K[X_d]^{SL_2(K)}=K[(KX_d)^{SL_2(K)}]$ are
$k_1=2$, $k_2=\cdots=k_p=1$ when $K[X_d]^{SL_2(K)}=K[(KX_d)^{SL_2(K)}]\cong K[X_{d-1}]$ and
$k_1=\cdots=k_p=1$ with the trivial action of $SL_2(K)$ on $KX_d$ (and the latter case is impossible because $G\cong SL_2(K)$ is a nontrivial subgroup
of $GL_d(K)$).
\end{example}

\subsection{Weitzenb\"ock derivations}
A linear operator $\delta$ of an algebra $A$ is a {\it derivation} if
\[
\delta(uv)=\delta(u)v+u\delta(v),\quad u,v\in A.
\]
If $\mathfrak V$ is a variety of algebras, then every mapping $\delta: X_d\to F_d({\mathfrak V})$ can be uniquely extended to a derivation of
$F_d({\mathfrak V})$ which we shall denote by the same symbol $\delta$. If $\delta$ is a nilpotent linear operator on $KX_d$, then the
induced derivation is called a {\it Weitzenb\"ock derivation}. Weitzenb\"ock \cite{W} proved that in the case of polynomial algebras the
{\it algebra of constants}
\[
K[X_d]^{\delta}=\{f(X_d)\in K[X_d]\mid \delta(f(X_d))=0\}
\]
is finitely generated. Details on the algebra of constants $K[X_d]^{\delta}$ can be found in the book by Nowicki \cite{Now}.
For varieties $\mathfrak V$ of unitary associative algebras (and $\delta\not=0$) the algebra
$F_d({\mathfrak V})^{\delta}$ is finitely generated if and only if $\mathfrak V$ does not contain the algebra $T_2(K)$ of $2\times 2$ upper triangular matrices,
see \cite{Dr4, DrG}. Up to a change of the basis of $KX_d$ the Weitzenb\"ock derivation $\delta$ is determined by the Jordan normal form $J(\delta)$ of
the linear operator $\delta$ acting on $KX_d$. Since $\delta$ acts nilpotently on $KX_d$, the matrix $J(\delta)$ consists of Jordan blocks with zero
diagonals.

\begin{proposition}\label{Weitzenboeck}
If $d>2$ and the Jordan normal form $J(\delta)$ of the Weitzenb\"ock derivation consists of less than $d-1$ blocks, then the algebra
$F_d({\mathfrak V})^{\delta}$ is not finitely generated for any variety $\mathfrak V$ containing the variety $\mathfrak L$.
\end{proposition}

\begin{proof}
Since $\alpha\delta$, $\alpha\in K$, is nilpotent on $KX_d$, it is a {\it locally nilpotent derivation} of $F_d({\mathfrak V})$, i.e.,
for every $f(X_d)\in F_d({\mathfrak V})$ there exists an $n\geq 1$ such that $(\alpha\delta)^n(f(X_d))=0$.
Hence
\[
\exp(\alpha\delta)=1+\frac{\alpha\delta}{1!}+\frac{(\alpha\delta)^2}{2!}+\cdots
\]
is a well defined linear automorphism of $F_d({\mathfrak V})$. It is well known that the group
\[
\{\exp(\alpha\delta)\mid \alpha\in K\}
\]
is isomorphic to the unipotent group $UT_2(K)$ and
\[
F_d({\mathfrak V})^{\delta}=F_d({\mathfrak V})^{UT_2(K)}.
\]
If the matrix $J(\delta)$ consists of $p$ blocks, then the dimension of the vector space $(KX_d)^{\delta}$ of the linear constants is equal to
the number of the blocks $p$. Reading carefully \cite[Proposition 6.5.1, p. 65]{Now} we can see that
\[
\text{transcend.deg}(K[X_d]^{\delta})=d-1
\]
which is larger than $p=\dim(KX_d)^{\delta}$.
Now the proof follows from Proposition \ref{base of examples} applied for $UT_2(K)\subset GL_d(K)$.
\end{proof}

\begin{remark} If in Proposition \ref{Weitzenboeck} the Jordan normal form of $\delta$ consists of $d-1$ blocks, then
the algebra of constants $K[X_d]^{\delta}$ is generated by linear constants. In this case we may assume that
$\delta(x_1)=x_2$ and $\delta(x_i)=0$ for $i=2,\ldots,d$.
It is easy to see that $F_d({\mathfrak L})^{\delta}$ is generated by $x_1x_2-x_2x_1,x_2,\ldots,x_d$.
We do not know how far can be lifted to $F_d({\mathfrak V})$ the finite generation property of the algebra of constants
and do not have a description of the varieties $\mathfrak V$ containing $\mathfrak L$ such that the algebra $F_d({\mathfrak V})^{\delta}$
is finitely generated.
\end{remark}

\end{document}